%% file: ArxivFinal.tex
\documentclass[11pt]{article}

\usepackage{enumerate}
\usepackage{indentfirst}
\usepackage{amsmath}
\usepackage{amssymb}
\usepackage{amsthm}
\usepackage[mathscr]{euscript}
\usepackage{mathtools}
\usepackage{relsize}
\usepackage{xparse}
\usepackage{tikz-cd}
\usepackage{macros}
\usepackage{xcolor}

\usepackage[hyphens]{url}
\usepackage{hyperref}
\hypersetup{breaklinks=true}

\usepackage{cite}

\newcommand{\Yb}{\mathbf{Y}}
\newcommand{\Cc}{\mathcal{C}}
\newcommand{\Pc}{\mathcal{P}}

\usepackage{tikz}
\usetikzlibrary{arrows,chains,matrix,positioning,scopes}

\makeatletter
\tikzset{join/.code=\tikzset{after node path={%
\ifx\tikzchainprevious\pgfutil@empty\else(\tikzchainprevious)%
edge[every join]#1(\tikzchaincurrent)\fi}}}
\makeatother
\tikzset{>=stealth',every on chain/.append style={join},
         every join/.style={->}}
\tikzstyle{labeled}=[execute at begin node=$\scriptstyle,
   execute at end node=$]

\title{Big pure mapping class groups are never perfect}
\author{George Domat \\ \; Appendix with Ryan Dickmann}
\date{\today}
  
\begin{document}

\maketitle
\begin{abstract}
    We show that the closure of the compactly supported mapping class group of an infinite-type surface is not perfect and that its abelianization contains a direct summand isomorphic to $\oplus_{2^{\aleph_{0}}}\Q$. We also extend this to the Torelli group and show that in the case of surfaces with infinite genus the abelianization of the Torelli group contains an indivisible copy of $\displaystyle\oplus_{2^{\aleph_{0}}}\Z$ as well. Finally we give an application to the question of automatic continuity by exhibiting discontinuous homomorphisms to $\Q$.
\end{abstract}


\section{Introduction}

Let $S$ be a connected, orientable, second-countable surface. The mapping class group, $\MCG(S)$, is the group of orientation-preserving homeomorphisms of $S$ up to homotopy, where all homeomorphisms and homotopies fix the (possibly empty) boundary of $S$ point-wise. When $S$ is infinite type, that is, when $\pi_{1}(S)$ is not finitely generated, we will often call $\MCG(S)$ a \emph{big} mapping class group. The pure mapping class group, $\PMCG(S)$, is the subgroup of $\MCG(S)$ consisting of elements that fix the ends of $S$. In the finite-type setting it is a classic result of Powell \cite{Pow1978} that $\PMCG(S)$ is perfect, that is, has a trivial abelianization, whenever $S$ has genus at least $3$. We will see that this is not the case in the infinite-type setting. 

Let $\PMCGc(S)$ denote the subgroup of $\PMCG(S)$ consisting of compactly-supported mapping classes. We prove the following when $S$ has more than one end. The one-ended case is proved with Ryan Dickmann in the attached appendix by applying the Birman Exact Sequence. 

\begin{customthm}{A} \label{mainthm}
    $\overline{\PMCGc(S)}$ is not perfect if $S$ is an infinite-type surface. 
\end{customthm}

This disproves Conjecture 5 in \cite{APV2020}. In \cite{APV2020} the authors show that once $S$ has at least two ends accumulated by genus there exist nontrivial homomorphisms from $\PMCG(S)$ to $\Z$ so that $\PMCG(S)$ cannot be perfect. The maps they build come from handleshifts and for genus 2 and greater they prove that the integral cohomology of the closure of the compactly-supported mapping classes is trivial. The authors in \cite{DP2020} prove the same for surfaces with a single genus. Note that we get nontrivial homomorphisms to $\Z$ from $\PMCG(S)$ when $S$ has genus $0$ for free by first taking a forgetful map to a sphere with finitely many punctures (see \cite{DP2020} for a discussion on this). Thus $\PMCG(S)$ also cannot be perfect when $S$ has genus $0$.

In \cite{PV2018} the authors prove that $\PMCG(S) = \overline{\PMCGc(S)}$ if and only if $S$ has at most one end accumulated by genus. Combining the previous work in \cite{APV2020} with our main theorem we see that big pure mapping class groups are \emph{never} perfect.

\begin{customthm}{B}
    $\PMCG(S)$ is not perfect if $S$ is an infinite-type surface.
\end{customthm}

We prove this theorem for surfaces with at least two ends in Section 6. This proof relies on the existence of curves that separate at least two ends of the surface. The final case of the surface with one end, the Loch Ness monster surface, is proved in the appendix with Ryan Dickmann. Here we  add a puncture to the Loch Ness monster and apply the Birman Exact Sequence. 

Now that we know that $\overline{\PMCGc(S)}$ is not perfect we can ask: What is $H_{1}(\PMCGcc{S};\Z)$? Throughout this paper when we refer to the homology of a group we refer to its homology as a discrete group so that $H_{1}(G;\Z)$ is exactly the abelianization of the group $G$. We make use of the tools involved in the proof of Theorem \ref{mainthm} to find an uncountable direct sum of $\Q$'s inside the abelianization. We can then apply tools from abelian group theory to conclude the following. 

\begin{customthm}{C} \label{divsub}
    Let $S$ be an infinite-type surface. $H_{1}(\overline{\PMCGc(S)};\Z) = \left(\oplus_{2^{\aleph_{0}}}\Q\right) \oplus B$ where all divisible subgroups of $B$ are torsion.
\end{customthm}

We can similarly find such a direct summand in the abelianization of the Torelli group, $\cI(S)$. However, in the Torelli group we can make use of the Johnson homomorphism to see that the abelianization also contains uncountably many indivisible copies of $\Z$ whenever $S$ has infinite genus.

\begin{customthm}{D} \label{torelli}
    Let $S$ be an infinite-type surface. Then $H_{1}(\cI(S);\Z) = \left(\oplus_{2^{\aleph_{0}}}\Q\right) \oplus B$ where all divisible subgroups of $B$ are torsion. If $S$ also has infinite genus, then $B$ contains a copy of $\oplus_{2^{\aleph_{0}}}\Z$.
\end{customthm}

Finally we provide two applications of Theorem \ref{divsub}. The first is to the question of automatic continuity of big mapping class groups and the second pertains to endomorphisms of pure mapping class groups.

\begin{customcor}{E} \label{discontinuous}
    Let $S$ be an infinite-type surface. There exists $2^{\mathfrak{c}}$ discontinuous homomorphism from $\PMCGcc{S}$ to $\Q$. 
\end{customcor}

This gives some progress towards Questions 2.4 and 2.6 in \cite{Mann2020} that ask for which infinite-type surfaces do the mapping class groups or pure mapping class groups have automatic continuity. 

Aramayona and Souto in \cite{AS2013} prove that if $S$ is a finite-type surface without boundary and with genus at least $4$ then every endomorphism of the pure mapping class group is in fact an isomorphism. We can now give examples where this is not the case in the infinite-type setting.

\begin{customcor}{F} \label{endomorphisms}
    Let $S$ be the surface with infinite genus, no boundary components, and one end. Then there exist uncountably many distinct nontrivial endomorphisms of $\MCG(S)$ that are not isomorphisms.
\end{customcor}

To obtain these maps we use the recent result in \cite{AoPV2021} that $\Q$ is a subgroup of $\MCG(S)$ for this surface. Then we simply first map $\MCG(S)$ to $\Q$ using the maps from Theorem \ref{discontinuous} and then map $\Q$ into $\MCG(S)$. 

Our proof of the main result uses the projection complex machinery of \cite{BBF2015} and \cite{BBFS2020}. Projection complexes have proven very useful in the setting of finite-type mapping class groups (see \cite{BBF2016}, \cite{BBF2019}, and the original two papers mentioned previously). Recently the authors in \cite{HRQ2020} make use of the projection complex machinery to study the question of which big mapping class groups admit nonelementary continuous actions on hyperbolic spaces. 


\tableofcontents


\section{Outline}

Here we give an outline of the proof of Theorem \ref{mainthm}. Powell in \cite{Pow1978} shows that when $S$ is closed and without boundary $\PMCG(S)$ is perfect once $S$ has genus at least 3. It was also shown in \cite{Harer1983} that this is true whenever $S$ has any (finite) number of punctures or boundary components. However, in the finite-type setting pure mapping class groups are \emph{not} uniformly perfect. In fact, the number of commutators needed to write a power of a Dehn twist grows linearly in the power \cite{EK2001}. This idea gives some intuition for the proof of the main theorem. Consider the mapping class $f$ on the surface $S$ with two ends accumulated by genus given by an infinite product of increasing powers of Dehn twists about disjoint separating curves. That is, $f = \prod_{i \in \Z} T_{\gamma_{i}}^{|i|}$ where $\gamma_{i}$ is the bi-infinite sequence of disjoint separating curves pictured in Figure \ref{fig:laddercurves}.

\begin{figure}
	\centering
	\def\svgwidth{\columnwidth}
	    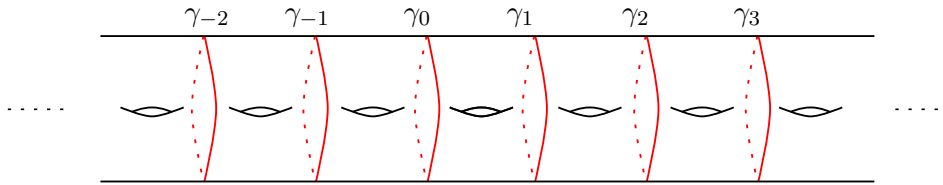
	    \caption{Curves used to find a counterexample to perfectness.}
    \label{fig:laddercurves}
\end{figure}

Now, if we approximate $f$ on larger and larger finite-type subsurfaces the number of commutators needed to write these approximations of $f$ grows; however, if $\overline{\PMCGc(S)}$ were perfect then we would be able to write $f$ as a finite product of commutators. The challenge is how do we actually build a contradiction using this intuition. 

We will build a quasimorphism on $\PMCGcc{S}$ for each curve $\gamma_{i}$ that ``measures" the twisting of $f$ about $\gamma_{i}$. However, quasimorphisms are always bounded on commutators, so we get a uniform upper bound on the value of $f$ for each of these quasimorphisms. This will yield a contradiction as $f$ twists more and more about each $\gamma_{i}$ as $i$ increases. 

This transforms the problem into building quasimorphisms with these properties. To do this we will use the projection complex machinery of \cite{BBF2015} and \cite{BBFS2020}. We will use $\gamma_{i}$ and the orbit of $\gamma_{i}$ under $\PMCGcc{S}$ to build a projection complex for each $i$. This will give an action of $\PMCGcc{S}$ on a quasi-tree with $T_{\gamma_{i}}$ acting as a WWPD element. We then use a generalization of the Brooks construction from \cite{BBF2016} to build our desired quasimorphisms.

Section 3 establishes the definitions and previous results we use. Sections 4 and 5 pertain to building the projection complexes and quasimorphisms we use. Section 6 proves that certain types of infinite products of Dehn twists about separating curves are nontrivial in $H_{1}(\PMCGcc{S};\Z)$. As a corollary we obtain Theorem \ref{mainthm}. Section 7 shows that we can similarly replace the infinite product of Dehn twists with infinite products of partial pseudo-Anosovs supported on disjoint subsurfaces and obtain the same result. This is important for proving Theorem \ref{divsub} in the low-genus case. Section 8 contains the proof of Theorem \ref{divsub}. Section 9 contains a discussion on the Torelli group and the proof of Theorem \ref{torelli}. Section 10 discusses an application to the question of automatic continuity by building discontinuous homomorphisms and contains the proof of Theorem \ref{discontinuous}. Section 11 contains a conversation and poses some questions about other possible elements in $H_{1}(\PMCGcc{S};\Z)$. Finally, the appendix with Ryan Dickmann give the proof of Theorem \ref{mainthm} in the case of the Loch Ness Monster (one-ended) surface.


\section{Background}

We will always assume that our surfaces are connected, orientable, second countable, and possibly with finitely many compact boundary components.

\subsection{Ends, classification, and exhaustions of infinite-type surfaces}

The \textbf{space of ends} of a surface $S$ is given by $\operatorname{Ends}(S) = \varprojlim_{K}(S\setminus K)$ where $K$ ranges over the compact subsets of $S$. It can be given a topology that is totally-disconnected, separable, and compact so that it is always homeomorphic to a subset of the Cantor set. We say an end is \textbf{accumulated by genus} if every open set in $S$ containing that end has infinite genus. We denote the set of ends accumulated by genus as $\operatorname{Ends}_{\infty}(S)$.

\begin{theorem} [Classification of Surfaces, \cite{Kerekjarto1923} \cite{Richards1963}]
    A surface, $S$, with finitely many compact boundary components is determined up to homeomorphism by the quadruple $(g,b,\operatorname{Ends}(S),\operatorname{Ends}_{\infty}(S))$, where $g \in \N \cup \{\infty\}$ is the genus of $S$, $b \in \N$ is the number of boundary components, and the pair $(\operatorname{Ends}(S),\operatorname{Ends}_{\infty}(S))$ is considered up to topological type.
\end{theorem}

Note that this classification subsumes the classical classification of finite-type surfaces.

\begin{definition}
    We say that an essential, simple closed curve $\gamma$ in a surface $S$ is \textbf{end separating} if it separates the space of ends of $S$. Likewise we say that a finite-type subsurface $B \subset S$ is \textbf{end separating} if $\partial B \setminus \partial S$ is a collection of essential, end-separating, simple closed curves. 
\end{definition}

We will make use of a modification of the notion of a principal exhaustion as defined in \cite{HMV2019}. First we recall that the \textbf{topological complexity} of a finite-type surface $S$ is $3g-3+b+n$ where $g$ is the genus of $S$, $b$ is the number of boundary components of $S$, and $n$ is the number of punctures of $S$. 

\begin{definition}
    Let $\{S_{i}\}$ be an increasing sequence of closed subsurfaces of $S$, an infinite-type surface. We say that $\{S_{i}\}$ is a \textbf{separating principal exhaustion} if $S = \bigcup_{i=1}^{\infty}S_{i}$ and for all $i$ it satisfies the following:
    \begin{enumerate}[(i)]
        \item 
            each $S_{i}$ is an end-separating surface,
        \item
            $S_{i}$ is contained in the interior of $S_{i+1}$, and
        \item 
            each component of $S_{i+1} \setminus S_{i}$ has topological complexity at least $6$. 
    \end{enumerate}
\end{definition}
A separating principal exhaustion always exists for any infinite-type surface with at least two ends. 

\subsection{Big mapping class groups}

For $S$ a surface, possibly with boundary, let $\Homeo_{\partial}^{+}(S)$ be the group of orientation-preserving homeomorphisms that fix the boundary pointwise. The \textbf{mapping class group}, $\MCG(S)$, is defined to be
\begin{align*}
    \MCG(S) = \Homeo_{\partial}^{+}(S)/\sim
\end{align*}
where two homeomorphisms are equivalent if they are isotopic relative to the boundary of $S$. When $S$ is of finite type, $\MCG(S)$ is discrete. In the infinite-type setting we equip $\Homeo_{\partial}^{+}(S)$ with the compact-open topology, this in turn induces the quotient topology on $\MCG(S)$. The \textbf{pure mapping class group}, $\PMCG(S)$, is the kernel of the action of $\MCG(S)$ on the space of ends of $S$ equipped with the subspace topology.

We say that $f \in \MCG(S)$ is \textbf{compactly supported} if $f$ has a representative that is the identity outside of a compact subset of $S$. The subgroup consisting of compactly supported mapping classes is denoted $\PMCGc(S) \subset \MCG(S)$. Note that any compactly-supported mapping class is in the subgroup $\PMCG(S)$. 

In \cite{APV2020} the authors decompose $\PMCG(S)$ as a semi-direct product of $\PMCGcc{S}$ and a group generated by handle shifts. 

\begin{theorem} [\cite{APV2020}, Corollary 4] \label{structure}
    $\PMCG(S) = \PMCGcc{S} \rtimes H$ where $\displaystyle H \cong \prod_{n-1} \Z$ with $n \in \N \cup \{\infty\}$ the number of ends of $S$ accumulated by genus and $H$ trivial if $n \leq 1$.
\end{theorem}

See \cite{PV2018} and \cite{APV2020} for the definition of a handle shift and a more thorough introduction to big mapping class groups. We note that if $S$ has only one end accumulated by genus, then $\PMCG(S) = \PMCGcc{S}$


\subsection{Projection complexes}

In this section we will review the projection complex machinery of \cite{BBF2015} and \cite{BBFS2020} that will be used to build an action of $\PMCGcc{S}$ on a quasi-tree.

A \textbf{quasi-tree} is a geodesic metric space that is quasi-isometric to a tree. Manning in \cite{Manning2005} gave the following equivalent characterization. A geodesic metric space X is a quasi-tree if and only if it satisfies the \textbf{bottleneck criterion}: There exists $\Delta\geq 0$ such that for any two points $x,y\in X$ the midpoint $z$ of a geodesic between $x$ and $y$ has the property that any path between $x$ and $y$ intersects the $\Delta$-ball about $z$. Here $\Delta$ is called the \textbf{bottleneck constant}.

The setup of the projection complex machinery begins with some given data: 
\begin{itemize}
    \item $\Yb$ a set,
    \item for each $Y \in \Yb$ an associated geodesic metric space $\Cc(Y)$, and
    \item for $X,Z \in \Yb$, with $X \neq Z$, a projection $\pi_{Z}(X) \subset \Cc(Z)$ from $X$ to $Z$.
\end{itemize}
Then we define $d_{Y}(X,Z) = \diam(\pi_{Y}(X) \cup \pi_{Y}(Z))$ for $X,Y,Z \in \Yb$. 

\begin{definition}
    The collection $\{(\Cc(Y),\pi_{Y})\}_{Y\in\Yb}$ satisfies the \textbf{projection axioms} for a \textbf{projection constant} $\theta \geq 0$ if
    \begin{itemize}
        \item [(P0)] $\diam(\pi_{Y}(X)) \leq \theta$ when $X \neq Y$,
        \item [(P1)] if $X,Y,Z$ are distinct and $d_{Y}(X,Z) > \theta$ then $d_{X}(Y,Z) \leq \theta$,
        \item [(P2)] if $X \neq Z$, the set $\{Y \in \Yb \vert d_{Y}(X,Z) > \theta\}$ is finite. 
    \end{itemize}
\end{definition}

Given such a collection and a constant $K>0$ one can define the \textbf{projection complex} $\Pc_{K}(\Yb)$ to be the graph with vertex set $\Yb$ and edges joining $X,Z \in \Yb$ whenever $d_{Y}(X,Z)<K$ for all $Y \in \Yb \setminus\{X,Z\}$. We then get the \textbf{blown-up projection complex} $\Cc_{K}(\Yb)$ by replacing each vertex $Y \in \Yb$ with $\Cc(Y)$ and joining points in $\pi_{X}(Z)$ with points in $\pi_{Z}(X)$ by an edge of length $L=L(K)$ whenever $X$ and $Z$ have an edge between them in $\Pc_{K}(\Yb)$. Technically $\Cc_{K}(\Yb)$ depends on a choice of $L$ and $K$ but we will fix $L$ as a function of $K$. If $K$ is sufficiently large then each $\Cc(Y)$ will be isometrically embedded in $\Cc_{K}(\Yb)$ (\cite{BBF2015}, Lemma 4.2). 

We say that a group $G$ acting on $\Yb$ preserves the projection structure if for every $Y \in \Yb$ and $g \in G$ there are isometries $F_{g}^{Y}:\Cc(Y) \rightarrow \Cc(g(Y))$ so that 
\begin{enumerate}[(i)]
    \item $F_{g'}^{g(Y)}F_{g}^{Y}=F_{g'g}^{Y}$ for all $g,g' \in G, Y \in \Yb$ and 
    \item $g(\pi_{Y}(X)) = \pi_{g(Y)}(g(X))$ for all $g \in G$ and $X,Y \in \Yb$. 
\end{enumerate}
If $G$ acts in this way it preserves the projection distances and acts naturally on $\Pc_{K}(\Yb)$ and $\Cc_{K}(\Yb)$ by isometries. 

Provided that $K$ is large enough this construction yields a quasi-tree.

\begin{theorem}[\cite{BBF2015}, Theorem 3.16, Theorem 4.14, and Theorem 4.17] \label{QT}
    If $\{(\Cc(Y),\pi_{Y})\}_{Y\in\Yb}$ satisfies the projection axioms with projection constant $\theta$ and $K>3\theta$ then
    \begin{enumerate}[(i)]
    \item 
        $\Pc_{K}(\Yb)$ is a quasi-tree.
    \item
        If all $\Cc(Y)$ are quasi-trees with uniform bottleneck constants for all $Y \in \Yb$, $\Cc_{K}(\Yb)$ is a quasi-tree. Furthermore, the bottleneck constant of $\Cc_{K}(\Yb)$ depends only on the bottleneck constants of the $\Cc(Y)$ and the projection constant. 
    \item
        If all $\Cc(Y)$ are $\delta$-hyperbolic with the same $\delta$ then $\Cc_{K}(\Yb)$ is hyperbolic with hyperbolicity constant depending only on $\delta$ and the projection constant. 
    \end{enumerate}
\end{theorem}

We will primarily be utilizing the blown-up projection complex. For our uses we only need that $K$ is sufficiently larger than the projection constant and so we will often drop the $K$ and simply write $\Pc(\Yb)$ and $\Cc(\Yb)$ for the projection complex and blown-up projection complex, respectively.


\subsection{Curve graphs and projections}

We will be using the curve graphs and subsurface projections as defined in \cite{MM1999} and \cite{MM2000} to build our projection complex. Recall that the curve graph of an orientable surface with boundary, $S$, is the graph $\Cc(S)$ with vertices homotopy classes of simple closed curves and edges between any two classes that can be realized disjointly on $S$. We can then define projections between curve graphs of essential subsurfaces of $S$. If $Y$ and $Z$ are essential subsurfaces with $\partial Z$ intersecting $Y$, then $\partial Z \cap Y$ is a collection of curves and arcs in $Y$. For each of these arcs one can perform surgery, in potentially two different ways, with $\partial Y$ to close it up to a curve in $Y$. Then we define $\pi_{Y}(Z) \subset \Cc(Y)$ to be the union of all curves and closed up arcs coming from $\partial Z \cap Y$. This gives a definition for whenever our subsurfaces have negative Euler characteristic; however, we will also be concerned with annular subsurfaces and projections between them. 

We now define the curve graph for a simple closed curve in $S$, or equivalently an annular subsurface of $S$. Fix a hyperbolic metric on the interior of $S$. If $\gamma$ is an essential non-peripheral simple closed curve let $X_{\gamma}$ be the annular cover of $S$ corresponding to $\gamma$. Now let $\Cc(\gamma)$ be the graph with vertices complete geodesics in $X_{\gamma}$ that cross the core curve and an edge between any pair of geodesics that are disjoint. In \cite{MM2000} it is shown that $\Cc(\gamma)$ is quasi-isometric to $\Z$. In fact, there is always a $(1,2)$-quasi-isometry, regardless of the topological type of the underlying surface or curve.

For $\beta$ another essential non-peripheral simple closed curve intersecting $\gamma$ we define the projection $\pi_{\gamma}(\beta)$ to be the components in $X_{\gamma}$ that intersect the core curve of the preimage of the geodesic representative of $\beta$ in $S$. 

We say that two subsurfaces $Y$ and $Z$ \textbf{overlap} if $i(\partial Y, \partial Z) \neq 0$ where $i$ denotes the geometric intersection number and if $\gamma$ is an essential non-peripheral simple closed curve we say that the boundary of the corresponding annular subsurface is simply $\gamma$. Note that projections between subsurfaces are only defined when they overlap. 

We can now define distances between subsurface projections. For $X,Y,Z$ three overlapping subsurfaces (potentially annuli) define
\begin{align*}
    d_{Y}(X,Z) = \diam_{\Cc(Y)}(\pi_{Y}(X),\pi_{Y}(Z)).
\end{align*}

For $\beta$ any curve intersecting $\gamma$ transversely we have that $d_{\gamma}(T_{\gamma}^{n}(\beta),\beta) = 2 + |n|$ for all $n \neq 0$. The additive factor of two comes from the fact that the Dehn twist in $S$ will affect every lift of $\gamma$ to $X_{\gamma}$ so that the lifts of $T_{\gamma}^{n}(\beta)$ are twisted an extra amount, causing it to pick up two additional intersections as pictured in Figure \ref{fig:DTlift}. This was also shown in \cite{MM2000}.

\begin{figure}
	   \centering
	   \def\svgwidth{\columnwidth}
	    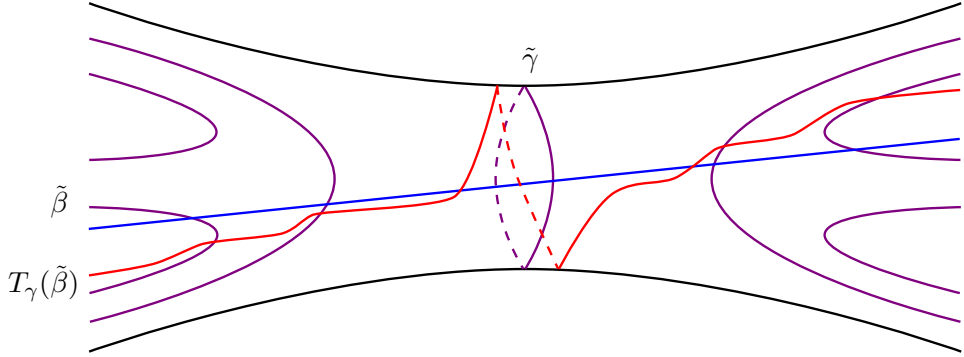
	       \caption{A schematic picture of the action of a Dehn Twist in the annular cover corresponding to $\gamma$ on the curve $\beta$ that intersects $\gamma$ once. The complete lift of $\gamma$ is given in purple, a lift of $\beta$ is blue, and the twist of this lift is in red.}
    \label{fig:DTlift}
\end{figure}

Now we have the following lemma that follows from work in \cite{MM2000} and the Behrstock inequality, \cite{Beh2006}. The explicit bound of 10 for the Behrstock inequality follows from a proof of Leininger as recorded in \cite{Mangahas2010}.

\begin{lemma}[\cite{BBF2015}, Section 5.1] \label{MCGproj}
    Let $\Yb$ be a collection of pairwise overlapping subsurfaces in a compact orientable finite-type surface $S$ such that $\chi(S)< 0$, possibly with finitely many punctures (compact after the punctures are filled in). Then $\{(\Cc(Y),\pi_{Y})\}_{Y\in\Yb}$, where $\Cc(Y)$ denotes the curve graph of $Y$, satisfies the projection axioms (P0)-(P2) with projection constant $\theta = 10$. 
\end{lemma}

This lemma also holds in the infinite-type setting.

\begin{lemma} \label{BMCGproj}
    Let $\Yb$ be a collection of pairwise overlapping finite-type subsurfaces in an orientable infinite-type surface $S$. Then $\{(\Cc(Y),\pi_{Y})\}_{Y\in\Yb}$, where $\Cc(Y)$ denotes the curve graph of $Y$, satisfies the projection axioms (P0)-(P2) with projection constant $\theta = 10$. 
\end{lemma}

\begin{proof}
    (P0) and (P1) follow exactly as in the finite-type setting. For (P2), if $X,Y \in \Yb$ we let $A$ be the smallest finite-type subsurface of $S$ that contains both $X$ and $Y$. Note that if $Z$ is a third subsurface and $Z$ is not contained within $A$ then there is some curve $\gamma$ in $Z$ disjoint from $\partial X \cap Z$ and $\partial Y \cap Z$, so it suffices to only consider subsurfaces $Z$ contained in $A$. Then (P2) holds due to the fact that it holds for $X$ and $Y$ as subsurfaces of $A$. 
\end{proof}


\subsection{WWPD elements}

The construction above gives an action of a group $G$ on a $\delta$-hyperbolic space, $\Cc_{K}(\Yb)$. $\Cc_{K}(\Yb)$ is $\delta$-hyperbolic because the hyperbolicity constant of a curve graph of a finite-type surface is independent of the topological type of the surface. This is due to \cite{Aougab2013}, \cite{Bowditch2014}, \cite{CRS2015}, and \cite{HPW2015}, independently. This action will not be proper; however, we will still have some control on how certain elements of $G$ act. We say that a \textbf{virtual quasi-axis} of an element $g \in G$ is a quasi-axis of a power of $g$. 

\begin{definition}
    Let $(G,X,g,C)$ be a quadruple with
    \begin{enumerate}[(i)]
        \item X a $\delta$-hyperbolic graph,
        \item $G$ a group acting on $X$ by isometries,
        \item $g \in G$ a hyperbolic isometry of $X$ with fixed points $x_{\pm \infty}$ at infinity, and
        \item $C < G$ a subgroup that fixes $x_{\pm \infty}$ pointwise; equivalently, for every virtual quasi-axis $\gamma$ the orbit $C\gamma$ is contained in a Hausdorff neighborhood of $\gamma$ and no element of $C$ flips the ends.
    \end{enumerate}
    Given such a quadruple, if there exists some $\xi >0$ and quasi-axis $\ell$ for $g$ such that for every $h \in G \setminus C$ we have that the projection of $h\cdot \ell$ to $\ell$ has diameter $\leq \xi$ we say that the quadruple $(G,X,g,C)$ satisfies \textbf{WWPD} with constant $\xi$. We also say that the element $g$ is \textbf{WWPD} if it is part of such a quadruple.
\end{definition}

WWPD elements will be important in our construction of quasimorphisms.


\subsection{Quasimorphisms}

\begin{definition}
    A \textbf{quasimorphism} of a group $G$ is a function $F:G \rightarrow \R$ such that 
    \begin{align*}
        D(F) \defeq \sup_{g,h \in G} \left|F(gh)-F(g)-F(h)\right| < \infty.
    \end{align*}
    We say $D(F)$ is the \textbf{defect} of $F$. We say that $F$ is \textbf{antisymmetric} if $F(g^{-1}) = -F(g)$ for all $g \in G$.  
\end{definition}

Note that any antisymmetric quasimorphism $F:G \rightarrow \R$ is bounded on commutators; that is,
\begin{align*}
    |F([g,h])| \leq 3D(F)
\end{align*}
for all $g,h \in G$. Also, if $g \in G$ can be written as a product of $C$ commutators then we have
\begin{align*}
    |F(g)| \leq 3C\cdot D(F) + (C-1)D(F) \leq 4C \cdot D(F).
\end{align*}
These inequalities follow directly from applying the definition of an antisymmetric quasimorphism. 

In \cite{BBF2016} the authors generalize the classical Brooks construction \cite{Bro1981} to the setting of groups acting on quasi-trees with WWPD elements. 

\begin{proposition} [\cite{BBF2016}, Proposition 3.1] \label{qm}
    For every $\Delta > 0$ there is $M=M(\Delta)$, a fixed multiple of $\Delta + 1$, such that the following holds. Let $(G,Q,g,C)$ satisfy WWPD with constant $\xi_{g}$ where $Q$ is a quasi-tree with bottleneck constant $\Delta$ and assume that $\tau_{g}\geq \xi_{g}+M$ where $\tau_{g}$ is the translation length of $g$. Then there is an antisymmetric quasimorphism $F:G \rightarrow \R$ such that 
    \begin{enumerate}[(i)]
        \item $D(F) \leq 12$, and
        \item $F$ is unbounded on the powers of $g$. In fact, $F(g^{n}) \geq \frac{n}{2} -1$. 
    \end{enumerate}
\end{proposition}

The proposition as stated in \cite{BBF2016} has more consequences but we've only listed the two that we will take advantage of. 

\begin{figure}
	   \centering
	   \def\svgwidth{\columnwidth}
	    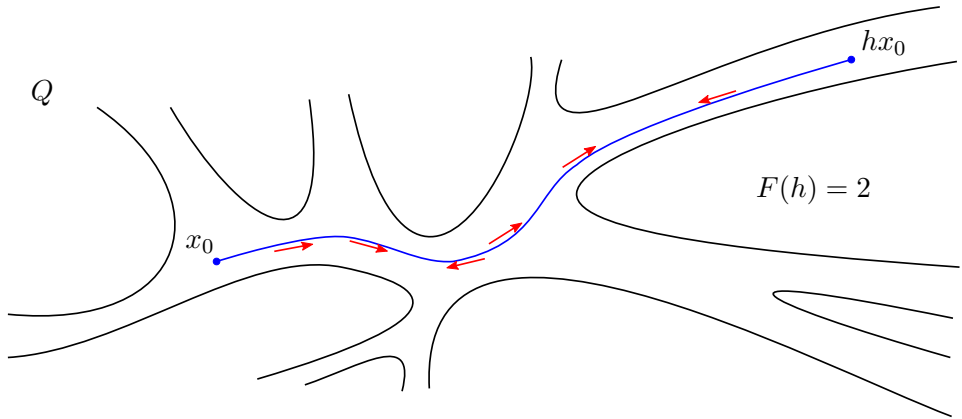
	       \caption{A schematic pictue of how the quasimorphism $F$ is defined. $F(h)$ counts the (oriented) copies of $w$ (in red) that appear along a geodesic from $x_{0}$ to $hx_{0}$. }
    \label{fig:QMexample}
\end{figure}

We give an informal description of the map $F$ and direct the reader to 
\cite{BBF2016} for more detail. We begin by picking a basepoint $x_{0} \in Q$ that is moved minimally by $g$. Let $w = [x_{0},g x_{0}]$ be an oriented geodesic. $F(h)$ is then defined by counting the \emph{non-overlapping, oriented copies} of $w$ that appear along $[x_{0},h x_{0}]$. Here a \emph{copy} of $w$ is a $G$-translate of the segment $w$. See Figure \ref{fig:QMexample} for an example of the map $F$. This is analogous to the quasimorphism defined in \cite{Bro1981} for the free group. In the quasi-tree case some extra care is needed in formally defining this count and in showing that the resulting map is a quasimorphism. 


\section{Building projection complexes}

Let $S$ be an infinite-type surface with more than one end. We will build projection complexes out of the $\PMCGcc{S}$-orbits of either an end-separating curve in $S$ or an end-separating subsurface of $S$. 

\begin{lemma} \label{curveoverlap}
    Let $g,h \in \PMCGcc{S}$. Then 
    \begin{enumerate}[(i)]
        \item
            for any end-separating simple closed curve $\gamma$ in $S$ the translates $g(\gamma)$ and $h(\gamma)$ are either homotopic or overlap. That is, either $h(\gamma) = g(\gamma)$ or $h(\gamma) \cap g(\gamma) \neq \emptyset$. 
        \item
            for any end-separating subsurface $B \subset S$ the translates $g(B)$ and $h(B)$ are either homotopic or overlap. That is, either $h(B)= g(B)$ or $i(\partial(h(B)),\partial(g(B))) \neq 0$. 
    \end{enumerate}
\end{lemma}

\begin{proof}
    We start by proving (i). Without loss of generality we will show that $\gamma \cap g(\gamma) \neq \emptyset$ for all $g \in \PMCGcc{S}$ that do not fix $\gamma$ up to homotopy. Since $g \in \PMCGcc{S}$ there is some $g' \in \PMCGc(S)$ such that $g'(\gamma) = g(\gamma)$. There is a finite-type subsurface $K \subset S$ such that $\gamma,g(\gamma)$, and $\supp(g')$ are contained in $K$. Thus we can realize $g'$ as a pure mapping class of the subsurface $K$. Now since $\gamma$ separates the boundary curves and/or the punctures of $K$ and $g'$ is a pure mapping class of $K$ the only way for $g'(\gamma)$ to be disjoint from $\gamma$ would be if $g'$ maps one of the complementary components of $\gamma$ into itself. Since $g'$ fixes the topological types and boundary components of each of the complementary components of $\gamma$ this would only be possible if $\gamma$ and $g'(\gamma)$ were homotopic, a contradiction to our assumption. Thus we conclude that $\gamma \cap g'(\gamma) \neq \emptyset$, or equivalently, $\gamma \cap g(\gamma) \neq \emptyset$.
    
    For (ii) we apply (i) to each of the curves in $\partial B$. 
\end{proof}

We can now apply Lemma~\ref{BMCGproj} and Theorem~\ref{QT} to obtain an action of $\PMCGcc{S}$ on a graph that is either a quasi-tree or $\delta$-hyperbolic. 

\begin{proposition} \label{bprojcomplex}
    Let $S$ be an infinite-type surface and let $\Yb = \{g(A)\vert g \in \PMCGcc{S}\}$ where $A$ is either an end-separating curve on $S$ or an end-separating subsurface of $S$. Then $\PMCGcc{S}$ acts on a quasi-tree, the projection complex $\Pc_{A}(\Yb)$ corresponding to $\{\Cc(Y),\pi_{Y}\}_{Y \in \Yb}$, where the bottleneck constant for $\Pc_{A}(\Yb)$ is independent of the surface $S$ or $A$. Furthermore, $\PMCGcc{S}$ also acts on the blown up projection complex $\Cc_{A}(\Yb)$. When $A$ is an end-separating curve $\Cc_{A}(\Yb)$ is again a quasi-tree and when $A$ is an end-separating subsurface $\Cc_{A}(\Yb)$ is $\delta$-hyperbolic. In either case the respective bottleneck constant or $\delta$ is independent of $S$ and $A$. 
\end{proposition}

\begin{remark}
The independence of all of the constants follows from the fact that they only depend on two quantities: 
\begin{itemize}
    \item 
        The projection constant, which in all cases is $10$.
    \item
        The bottleneck constant of the curve graph in the case that $A$ is a curve, which is shown to be constant in \cite{MM2000}, and the hyperbolicity constant of the curve graph when $A$ is a finite-type surface, which is shown to be independent of topological type in \cite{Aougab2013}, \cite{Bowditch2014}, \cite{CRS2015}, and \cite{HPW2015}, independently.
\end{itemize}
\end{remark}


\section{Constructing quasimorphisms}

In this section we construct quasimorphisms on $\PMCGcc{S}$ that will ``see" elements that are nontrivial in $H_{1}(\PMCGcc{S};\Z)$. We will do this by showing that Dehn twists about end-separating curves are WWPD elements when acting on the projection complexes arising from Proposition \ref{bprojcomplex} and then applying Proposition \ref{qm}. Let $\gamma$ be an end-separating simple closed curve on $S$. 

\begin{lemma} \label{WWPDtwist}
    $(\PMCGcc{S},\Cc_{\gamma}(\Yb),T_{\gamma},\Stab(\gamma))$ satisfies WWPD with constant $\xi$ depending only on the projection constant of $\Cc_{\gamma}(\Yb)$ and with translation length $1$. Furthermore, for any power $n>0$, $(\PMCGcc{S},\Cc_{\gamma}(\Yb),T_{\gamma}^{n},\Stab(\gamma))$ also satisfies WWPD with the same constant. 
\end{lemma}

To prove this we need to make use of the fact that nearest-point projections in the projection complex are uniformly close to the given projections.

\begin{proposition} [\cite{BBF2015}, Corollary 4.10] \label{proj}
    For every $Z \in \Yb$ the nearest-point projection $\Cc(\Yb) \rightarrow \Cc(Z)$ is coarsely Lipschitz and the image of $\Cc(Y)$ for $Y \neq Z$ is in a uniform neighborhood of the bounded set $\pi_{Z}(Y)$. The uniform bound is a function of the projection constant. 
\end{proposition}

\begin{proof} [Proof of Lemma \ref{WWPDtwist}]
    We first note that $T_{\gamma}$ acts hyperbolically with translation length $1$. Indeed, within $\Cc(\gamma)$ we have that the projection distances satisfy $d_{\gamma}(T_{\gamma}^{n}(\alpha),\alpha) = 2+ |n|$ for all $n \neq 0$ where $\alpha$ is some curve in $S$ that intersects $\gamma$ transversely. $\Cc(\gamma)$ is then isometrically embedded within $\Cc(\Yb)$. 
    
    Fix a quasi-axis $\ell \subset \Cc(\gamma)$ for $T_{\gamma}$. For any $h \in \PMCGcc{K} \setminus \Stab(\gamma)$ we have that $h$ must move $\Cc(\gamma)$ to some other $\Cc(h(\gamma))$. Thus by Proposition \ref{proj} the diameter of the nearest-point projection of $h \cdot \ell$ to $\ell$ is bounded by the nearest point projection of a uniform neighborhood of $\pi_{\Cc(\gamma)}(\Cc(h\cdot\gamma))$ to $\ell$. This in turn is uniformly bounded by a function of the projection constant. 
\end{proof}

We can now apply Proposition \ref{qm} to this construction. By making an appropriate choice of basepoint in the construction of our quasimorphism we can gain control over the value of the quasimorphism on group elements that are sufficiently ``independent" of our Dehn twist. 

\begin{lemma} \label{qmh}
    Suppose $(\PMCGcc{S},\Cc_{\gamma}(\Yb),T_{\gamma}^{n},\Stab(\gamma))$ is as in Lemma \ref{WWPDtwist}. If $h \in \PMCGcc{S}$ fixes $\Cc(\gamma)$ and $\Cc(\gamma')$ for some $\gamma' \in \Yb$ with $\gamma' \neq \gamma$ then the quasimorphism $F$ obtained via Proposition \ref{qm} (when $n$ is sufficiently large) can be chosen to be bounded on $h$. Furthermore, if $n$ is greater than the projection constant, $10$, of $\Cc_{\gamma}(\Yb)$, then $F(h) = 0$. 
\end{lemma}

\begin{proof}
    If $h$ preserves $\Cc(\gamma)$ and $\Cc(\gamma')$ then it must preserve the projections $\pi_{\gamma}(\gamma')$ and $\pi_{\gamma'}(\gamma)$. These projections are bounded in diameter by the projection constant, $10$. Now we can pick our basepoint, $x_{0}$, in the construction of $F$ to be in the set $\pi_{\gamma}(\gamma')$ so that $d_{\Cc_{\gamma}(\Yb)}(x_{0},h(x_{0})) \leq 10$. We conclude that $|F(h)| \leq \frac{10}{n}$ since the translation length of $T_{\gamma}^{n}$ is $n$. The furthermore statement follows from the fact that $F$ is integer valued. 
\end{proof}


\section{Proof of Theorem \ref{mainthm}}

We will actually prove a stronger theorem than stated in the introduction. 

\begin{theorem} \label{mainthm2}
    Let $S$ be an infinite-type surface with more than one end, $\Gamma = \{\gamma_{i}\}_{i\in \N}$ and $\Gamma' = \{\gamma'_{i}\}_{i\in\N}$ be two collections of disjoint end-separating curves that eventually leave every compact set so that $\gamma_{i}' \neq \gamma_{i}$ is a translate of $\gamma_{i}$ by a compactly-supported mapping class and $\gamma_{i}' \cap \gamma_{j} = \emptyset$ for all $i\neq j$, and $A = \{a_{i}\}_{i \in \N}$ be an unbounded sequence of natural numbers. Then the mapping class
    \begin{align*}
        f_{\Gamma,A} \defeq \prod_{i=1}^{\infty} T_{\gamma_{i}}^{a_{i}} \in \overline{\PMCGc(S)}
    \end{align*}
    cannot be written as a product of commutators in $\overline{\PMCGc(S)}$. The same also holds for products $\phi f_{\Gamma,A}$ where $\phi \in \PMCGcc{S}$ is a mapping class that fixes $\gamma_{i}$ and $\gamma_{i}'$ for infinitely many $i$. 
\end{theorem}

We can first note that since the $\gamma_{i}$ are disjoint and eventually leave every compact set, maps of the form $f_{\Gamma,A}$ are indeed defined and contained in $\overline{\PMCGc(S)}$. 

\begin{proof} [Proof of Theorem \ref{mainthm2}]
    
The proof is a direct application of the following lemma. 

\begin{lemma} \label{contradiction}
    For every $C > 0$, there exists an integer $N_{0} > 0$ with the following property: If $N>N_{0}$, $\gamma$ is an end-separating curve, and $h \in \PMCGcc{S}$ such that $h$ fixes $\gamma$ as well as some $\PMCGcc{S}$-translate of $\gamma$, then $g=hT^{N}_{\gamma}$ cannot be written as a product of $C$ commutators. 
\end{lemma}

\begin{proof}
    Let $\xi$ be the WWPD constant coming from Lemma \ref{WWPDtwist} and $M$ be the constant coming from Proposition \ref{qm} when we apply it to the quasi-tree arising from a projection complex as in Proposition~\ref{bprojcomplex}. Note that $\xi$ and $M$ depend only on the projection constant of 10 coming from Lemma \ref{BMCGproj}. In particular, they do not depend on $S$ or $\gamma$. 
    
    Next we let $N_{1} = \max\{M+\xi,11\}$, $N_{2} = 2(48C+25)$, and $N_{0} = N_{1}N_{2}$. Suppose $g = hT_{\gamma}^{N}$ as in the statement of the lemma. Apply Proposition \ref{bprojcomplex} to $S$ and the curve $\gamma$ to get an action of $\PMCGcc{S}$ on the quasi-tree $\Cc_{\gamma}(\Yb)$. Denote the length metric on $\Cc_{\gamma}(\Yb)$ by $d_{C_{\gamma}}$. By Lemma \ref{WWPDtwist}, $(\PMCGcc{S},\Cc_{\gamma}(\Yb),T_{\gamma}^{N_{1}},\Stab(\gamma))$ satisfies WWPD with constant $\xi$ and $T_{\gamma}^{N_{1}}$ has translation length $N_{1} > 10$ and $N_{1} \geq M + \xi$. By our choice of $N_{1}$ we can apply Proposition \ref{qm} and Lemma \ref{qmh} to $(\PMCGcc{S},\Cc_{\gamma}(\Yb),T_{\gamma}^{N_{1}},\Stab(\gamma))$ in order to build a quasimorphism $F: \PMCGcc{S} \rightarrow \R$ with basepoint $x_{0}$ so that $F(h) = 0$. Thus we see that 
    \begin{align*}
        |F(hT_{\gamma}^{N}) - F(T_{\gamma}^{N})| < 12.
    \end{align*}
    
    Now write $N = AN_{1} + B$ for $A \geq N_{2}$ and $B<N_{1}$. Note that $B < N_{1}$ so that $d_{\Cc_{\gamma}}(x_{0},T_{\gamma}^{B}x_{0}) < d_{\Cc_{\gamma}}(x_{0},T_{\gamma}^{N_{1}}x_{0})$ and hence $F(T_{\gamma}^{B}) = 0$. Then we have
    \begin{align*}
        |F(T_{\gamma}^{AN_{1}+B}) - F((T_{\gamma}^{N_{1}})^{A}) - F(T_{\gamma}^{B})| &= |F(T_{\gamma}^{AN_{1}+B}) - F((T_{\gamma}^{N_{1}})^{A})| < 12.
    \end{align*}

    We chose $N_{2}$ so that $F(T_{\gamma}^{N_{1}})^{A} > 48C + 24$ by Proposition \ref{qm}(ii). Thus we see that
    \begin{align*}
        F(T_{\gamma}^{N}) > 48C + 12,
    \end{align*}
    and 
    \begin{align*}
        F(g) = F(hT_{\gamma}^{N}) > 48C.
    \end{align*}
    
    If $g$ could be written as a product of $C$ commutators than we would have $F(g) \leq 48C$, contradicting the lower bound found above.
\end{proof}

To finish the proof of our theorem we simply note that since $A$ is unbounded, for any $C>0$ we can always write $f_{\Gamma,A} = hT_{\gamma_{i}}^{a_{i}}$ with $a_{i} > N_{0}(C)$ coming from Lemma \ref{contradiction}. Here $h$ will be the product of all Dehn twists appearing in $f_{\Gamma,A}$ other than the twists about $\gamma_{i}$ and so satisfies the conditions of Lemma \ref{contradiction}. For the final claim we simply include $\phi$ into the expression for $h$.  

\end{proof}

Finally, we can conclude Theorem \ref{mainthm} from Theorem \ref{mainthm2} provided that the families of curves as in the statement of Theorem \ref{mainthm2} always exist. To find such a family of curves we can fix a separating principal exhaustion of a given infinite-type surface $S$ (with at least two ends) and take $\Gamma = \{\gamma_{i}\}_{i\in\N}$ to be a choice of one boundary curve of each $S_{i}$ in the exhaustion. Now for each $i$ pick a curve $\alpha_{i}$ in $S_{i+1} \setminus S_{i-1}$ that intersects $\gamma_{i}$. Set $\gamma_{i}'= T_{\alpha_{i}}(\gamma_{i})$. These collections of curves $\Gamma$ and $\Gamma'$ satisfy the hypotheses of Theorem \ref{mainthm2}. This proves Theorem \ref{mainthm} provided that $S$ has at least two ends. The one-ended case is proved in the appendix by applying the Birman exact sequence.

\section{Pseudo-Anosovs on disjoint subsurfaces} 

Now we see that the proof of Theorem \ref{mainthm2} also works when we replace the Dehn twists by pseudo-Anosovs supported on homeomorphic disjoint subsurfaces. This version will be used in the following section to prove that $H_{1}(\overline{\PMCGc(S)};\Z)$ contains an uncountable direct sum of $\Q$'s when $S$ has genus less than $3$. 

\begin{theorem} \label{mainthm3}
    Let $S$ be an infinite-type surface with at least two ends, $\mathcal{B} = \{B_{i}\}_{i\in\N}$ be a collection of disjoint subsurfaces of $S$, each of which is end separating and is homeomorphic to some fixed finite-type surface $B$ of topological complexity at least 5 and $A = \{a_{i}\}_{i\in \N}$ be an unbounded sequence of natural numbers. Suppose that $f\in \PMCG(B)$ is a pseudo-Anosov and let $f_{i} \in \PMCGc(S)$ be the mapping class that is equal to $f$ on $B_{i}$ and the identity outside of $B_{i}$. Then the mapping class 
    \begin{align*}
        f_{\mathcal{B},A} \defeq \prod_{i=1}^{\infty}f_{i}^{a_{i}} \in \PMCGcc{S}
    \end{align*}
    cannot be written as a product of commutators in $\PMCGcc{S}$. The same also holds for any mapping class of the form $\phi f_{\mathcal{B},A}$ where $\phi \in \PMCGcc{S}$ fixes $B_{i}$ for infinitely many $i$. 
\end{theorem}

To prove this we want to follow the same steps used in the proof above. Proposition \ref{bprojcomplex} gives an action of $\PMCGcc{S}$ on a blown-up projection complex built out of the curve graphs of the $\PMCGcc{S}$-orbit of $B_{i}$ for each $i$. Just as above we will build quasimorphisms on $\PMCGcc{S}$ for each $B_{i}$. For now we assume that $i$ is fixed and abuse notation to write $B=B_{i}$. 

Next we have to do two things. We have an action of $\PMCGcc{S}$ on a $\delta$-hyperbolic space, but we really need an action on a quasi-tree. Also we need to see that a pseudo-Anosov, $f$, supported on $B$ is a WWPD element for this action. This will all follow from work in \cite{BF2002}, \cite{BBF2015}, and \cite{BBF2016}.

The first step is to see that $f$ acts as a WWPD element on $\Cc_{B}(\Yb)$. The following proposition informs us that $f$ is a WPD element for the action of $\PMCG(B)$ on the curve graph of $B$. Since curve graphs are $\delta$-hyperbolic we can take as a definition for WPD to be that $f$ is a WWPD element together with the extra condition that the subgroup $C$ in the definition of WWPD is virtually cyclic. 

\begin{proposition} [\cite{BF2002}, Proposition 11] \label{pAWPD}
    Let $A$ be a finite-type surface of topological complexity at least $5$. The action of $\MCG(A)$ on the curve graph of $A$ is WPD. In particular, every pseudo-Anosov is a WPD element.
\end{proposition}

Now Proposition 4.20 in \cite{BBF2015} tells us that this WPD element for the action on a single curve graph gives a WWPD element for the action on the entire blown up projection complex with WWPD constant depending only on the projection constant. Finally we can use the following proposition to upgrade our WWPD action on a $\delta$-hyperbolic graph to a WWPD action on a quasi-tree.

\begin{proposition}[\cite{BBF2016}, Proposition 2.9] \label{qtfromwwpd}
    Let $X$ be a $\delta$-hyperbolic graph and assume $(G,X,g,C)$ satisfies WWPD with constant $\xi = \xi_{g}^{X}$. Then there is an action of $G$ on a quasi-tree $Q$ such that:
    \begin{enumerate}[(i)]
        \item
            The bottleneck constant, $\Delta$, for $Q$ depends only on $\delta$ and $\xi$ and is bounded by a multiple of $\delta + \xi + 1$,
        \item
            $(G,Q,g,C)$ satisfies WWPD with $\xi_{g}^{Q}$ bounded by a multiple of $\delta+\xi+1$.
    \end{enumerate}
\end{proposition}

\begin{proof} [Sketch of Proof]
    We apply the projection complex construction again. Say two conjugates of $g$ are equivalent if they have parallel quasi-axes. Now for each equivalence class we take the union of the quasi-axes of its members. This is a quasi-line with the subspace metric. Let the collection $\Yb$ be all of these quasi-lines. This collection will satisfy the projection axioms and so when we construct the projection complex we get a quasi-tree, $Q$. We get (i) by realizing that the projection constant used in the construction only depends on $\delta$ and $\xi$. (ii) follows again from Proposition 4.20 in \cite{BBF2015}. 
\end{proof}

We collect these facts in the following lemma.

\begin{lemma}
    Let $S$ be an infinite-type surface and suppose that $B\subset S$ is an end-separating subsurface of topological complexity at least $5$. Given $f \in \PMCG(B)$ a pseudo-Anosov, there exists a subgroup $C < \PMCGcc{S}$ and quasi-tree $Q$ so that $(\PMCGcc{S},Q,f,C)$ satisfies WWPD with constant $\xi$. Furthermore, Q is a quasi-tree with bottleneck constant that does not depend on how $B$ embeds as a subsurface of $S$ and $(\PMCGcc{S},Q,f^{n},C)$ also satisfies WWPD with constant $\xi$ for any integer $n\geq1$. 
\end{lemma}

\begin{proof}
    We first apply Proposition~\ref{bprojcomplex} to obtain an action of $\PMCGcc{S}$ on a blown-up projection complex, $\Cc_{B}$, built out of copies of the curve graph of $B$. Let $C$ be the subgroup of $\Stab_{\PMCGcc{S}}(B)$ that fixes pointwise the points at infinity in the curve graph for $B$ fixed by $f$. Then since $f$ is a WPD element for the action on a single curve graph by Proposition~\ref{pAWPD}, we can apply Proposition 4.20 in \cite{BBF2015} to get that $(\PMCGcc{S},\Cc_{B},f,C)$ satisfies WWPD. Finally we can apply Proposition~\ref{qtfromwwpd} to replace $\Cc_{B}$ by a quasi-tree $Q$ so that $(\PMCGcc{S},Q,f,C)$ satisfies WWPD. 
    
    The furthermore statement follows from the fact that the these constants can be taken to depend only on the hyperbolicity constant of the curve graph of $B$ and the project constant for $\Cc_{B}$. 
\end{proof}

This lemma allows us to apply Proposition \ref{qm} by passing to a sufficiently large power, $n$, of $f$. Now we get an analogous result as in Lemma \ref{qmh}. 

\begin{lemma}
    Let $(\PMCGcc{S},Q,f^{n},C)$ be as above. If $h \in \PMCGcc{S}$ acts as the identity on $B$ then the quasimorphism $F$ obtained via Proposition \ref{qm} (when $n$ is sufficiently large) can be chosen to be trivial on $h$. 
\end{lemma}

\begin{proof}
    $h$ acts as the identity on $B$ and so it fixes $\Cc(B)$ pointwise in $\Cc_{B}(\Yb)$. Thus $h$ must fix pointwise a quasi-axis of $f$ in $\Cc(B)$. This quasi-axis is one of the objects used to construct the projection complex $Q$. Now when we build $F$ we can take the basepoint to be on this quasi-axis so that $F(h)=0$. 
\end{proof}

We can also follow the exact same proof for Lemma \ref{contradiction} to get a version in this setting. 

\begin{lemma} \label{contradictionpa}
    Let $S$ be an infinite-type surface and $f \in \PMCGcc{S}$ a partial pseudo-Anosov supported on an end-separating subsurface $B$ of $S$. For every $C>0$, there exists an integer $N_{0}>0$, dependent on $f$, with the following property: If $N > N_{0}$ and $h \in \PMCGcc{S}$ such that $h$ fixes $B$, then $g=hf^{N}$ cannot be written as a product of $C$ commutators. 
\end{lemma}

With all of these pieces the proof of Theorem \ref{mainthm3} follows exactly as the proof of Theorem \ref{mainthm2}. Note that the constant $N_{0}$ is dependent on $f$ as an element of $\PMCG(B)$. This is the reason that we take higher and higher powers of the \emph{same} pseudo-Anosov in the statement of Theorem \ref{mainthm3}. The result should also hold for higher and higher powers of different pseudo-Anosovs provided that their translation lengths fall in a bounded range. 


\section{Divisible subgroup of $H_{1}(\overline{\PMCGc(S)};\Z)$}

We have seen that for any infinite-type surface, $S$, with more than one end, $H_{1}(\overline{\PMCGc(S)};\Z)$ is nontrivial. Next we will use our main theorem to find a subgroup isomorphic to $\oplus_{2^{\aleph_{0}}}\Q$ within the abelianization.

\begin{definition} 
    An element $g$ of a group $G$ is said to be \textbf{divisible by $n$} if the equation $g = x^{n}$ has a solution in $G$. We say that $g$ is \textbf{divisible} if it is divisible by $n$ for all $n\in\N$. An abelian group is called \textbf{divisible} if every element is divisible.
\end{definition}

We first find a divisible element in the abelianization, then we construct uncountably many independent elements, and finally we combine these two constructions to prove Theorem \ref{divsub}. We adopt the notation that an over-bar represents the image of a mapping class in $H_{1}(\PMCGcc{S};\Z)$. 


\subsection{Constructing divisible elements}

We will follow the construction of Bogopolski and Zastrow for infinitely-divisible elements in the first homology of the Hawaiian Earring and Griffiths' space as seen in \cite{BZ2012}. We will need a slight modification when $S$ has genus less than $3$.

\subsubsection{$S$ has genus at least 3} \label{infdiv1}

We first consider the case that $S$ is an infinite-type surface of genus at least $3$ and with more than one end. 

Let $\{\gamma_{i}\}_{i\in\N}$ and $\{\gamma_{i}'\}_{i\in\N}$ be as in the statement of Theorem \ref{mainthm2}. 
Let 
\begin{align*}
    f = \prod_{j=1}^{\infty}T_{\gamma_{j}}^{j!}.
\end{align*}
Note that $f$ satisfies the hypotheses of Theorem \ref{mainthm2} so that $\bar{f}$ is nontrivial in $H_{1}(\overline{\PMCGc(S)};\Z)$. 

Since $S$ has genus at least $3$, each individual $T_{\gamma_{i}}$ can be thought of as a Dehn twist on a finite type surface of genus at least $3$. We can then apply the theorems of Powell and Harer to see that each $T_{\gamma_{i}}$ can be written as a product of commutators in $\PMCGc(S)$. Therefore, if we delete finitely many of the $T_{\gamma_{i}}$ from $f$, the resulting equivalence class in $H_{1}(\overline{\PMCGc(S)};\Z)$ will be unchanged. By deleting all of the occurrences of $T_{\gamma_{1}}$ in $\bar{f}$ we see that we can write $\bar{f}$ as a square in $H_{1}(\overline{\PMCGc(S)};\Z)$. Indeed, $\bar{f}=\bar{f'}\bar{f'}$ where $f'$ is given by
\begin{align*}
    f' = \prod_{k=2}^{\infty}T_{\gamma_{k}}^{\frac{k!}{2}}.
\end{align*}
We also verify that $\bar{f'}$ is nontrivial in $H_{1}(\overline{\PMCGc(S)};\Z)$ by Theorem \ref{mainthm2}. 

Similarly, for all $n \in \N$, by deleting all occurrences of $T_{\gamma_{1}},\ldots,T_{\gamma_{n}}$ from $\bar{f}$ we see that $\bar{f}$ is an $(n+1)$-th power in $H_{1}(\overline{\PMCGc(S)};\Z)$. Thus we see that $\bar{f}$ is divisible in $H_{1}(\overline{\PMCGc(S)};\Z)$.

\subsubsection{General case} \label{infdiv2}

If $S$ has genus less than $3$ we can no longer simply use Dehn twists because we no longer get for free that they can be written as a product of commutators. Instead we will run the same construction using a pseudo-Anosov on a punctured sphere that we can write as a commutator. Here we will need to make use of Theorem \ref{mainthm3}. 

\begin{remark}
    We note that this proof works in all cases regardless of the genus of the surface with a slight modification. The slight modification pertains to the use of six-times punctured spheres in what follows. If $S$ does not contain an infinite collection of end-separating six-times punctured spheres we could replace them with some other collection of homeomorphic subsurfaces of sufficiently large complexity. We have included the previous subsection because the proof is simpler in the case of Dehn twists and gives intuition for the following. 
\end{remark}

Suppose that $h$ is a pseudo-Anosov on a six-times punctured sphere (or similarly a sphere with six boundary components) that can be written as a product of commutators. Now since $S$ is infinite-type we can follow the same steps as at the end of Section 6 to find a collection $\mathcal{B} = \{B_{i}\}_{i \in \N}$ where each $B_{i}$ is end-separating and homeomorphic to a sphere with six boundary components. Let $h_{i}\in \PMCGc(S)$ be the mapping class that is $h$ on $B_{i}$ and the identity elsewhere. We can now apply the same exact construction as in the previous section with $h_{i}$ instead of $T_{\gamma_{i}}$ to get a divisible element. We apply Theorem \ref{mainthm3} to see that it is nontrivial in $H_{1}(\PMCGcc{S};\Z)$. Note that we used a six-times punctured sphere since Proposition \ref{pAWPD} required a surface with sufficiently large complexity. 

Now we just need to find a pseudo-Anosov on a six-times punctured sphere that can be written as a commutator. We will obtain such a pseudo-Anosov from the following lemma that is an application of Thurston's construction from \cite{Thurston1988} as stated in \cite{FM11}.

\begin{lemma} \label{thurstonconst}
    Suppose $\alpha$ and $\beta$ are curves that fill a finite-type surface $S$. Then the element $T_{\alpha}^{2}T_{\beta}^{2}T_{\alpha}^{-2}T_{\beta}^{-2}$ is a pseudo-Anosov in $\MCG(S)$. 
\end{lemma}

\begin{proof}
    Thurston's construction gives that there is a representation $\rho:\la T_{\alpha},T_{\beta} \ra \rightarrow \PSL(2,\R)$ given by 
    \begin{align*}
        T_{\alpha} \rightarrow 
            \begin{pmatrix} 1 & -i(\alpha,\beta) \\ 0 & 1 \end{pmatrix}
        \;\;\;\;\;
        T_{\beta} \rightarrow 
            \begin{pmatrix} 1 & 0 \\ i(\alpha,\beta) & 1 \end{pmatrix}.
    \end{align*}
    Furthermore, this representation has the property that $f \in \la T_{\alpha},T_{\beta} \ra$ is periodic, reducible, or pseudo-Anosov if and only if $\rho(f)$ is elliptic, parabolic, or hyperbolic, respectively. Finally, we note that two filling curves intersect at least once so that the element $\rho(T_{\alpha}^{2}T_{\beta}^{2}T_{\alpha}^{-2}T_{\beta}^{-2})$ has trace in absolute value greater than 2. 
\end{proof}

This lemma allows us to obtain our desired pseudo-Anosov by taking two curves that fill the six-times punctured sphere.


\subsection{Uncountably many independent elements} \label{sec:uncountably}

Let $S$ be any infinite-type surface with more than one end (of any genus). We will apply a trick used in \cite{RS2007} and \cite{Mann2020}. For each $a \in \R$ let $\Lambda_{a}$ be an infinite subset of $\N$ such that $\Lambda_{a} \cap \Lambda_{b}$ is finite for all $a \neq b$. We can obtain $\Lambda_{a}$ by putting $\N$ in bijection with $\Q$ and then letting $\Lambda_{a}$ be a sequence of rational numbers approximating $a$. 

Once again let $\{\gamma_{i}\}_{i\in\N}$ and $\{\gamma_{i}'\}_{i\in\N}$ be as in the statement of Theorem \ref{mainthm2}. For $a \in \R$ enumerate elements of $\Lambda_{a}$ as $\{a_{i}\}_{i \in \N}$ and let
\begin{align*}
    f_{a} \defeq \prod_{i = 1}^{\infty} T_{\gamma_{a_{i}}}^{i} \in \overline{\PMCGc(S)}.
\end{align*}
By Theorem \ref{mainthm2} $\bar{f_{a}}$ is nontrivial in $H_{1}(\overline{\PMCGc(S)};\Z)$ for all $a\in\R$. Note also that since $\Lambda_{a}\cap\Lambda_{b}$ is finite for any $a \neq b$, any finite product of such $f_{a}$ also satisfies the hypotheses of Theorem \ref{mainthm2} so that any finite product is also nontrivial in $H_{1}(\overline{\PMCGc(S)};\Z)$. Thus we have the following proposition.

\begin{proposition}
    Let $S$ by any infinite-type surface with more than one end. Then $H_{1}(\overline{\PMCGc(S)};\Z)$ contains an uncountable collection of independent elements. 
\end{proposition}

Note that we could have applied this same technique to products of powers of pseudo-Anosovs on subsurfaces as in Theorem \ref{mainthm3}. 


\subsection{Proof of Theorem \ref{divsub}} \label{proofofC}

We can now modify the construction of a divisible element to find uncountably many independent divisible elements. For $a \in \R$, let $\Lambda_{a}$ be as above. In the genus at least $3$ case, for each $a\in\R$ we construct $f_{a}$ as in Section \ref{infdiv1} except by using only twists in $\Lambda_{a}$. That is, $f_{a}$ is defined as:
\begin{align*}
    f_{a} = \prod_{j=1}^{\infty}T_{\gamma_{a_{j}}}^{j!}
\end{align*}

In the genus less than three case we do the same construction but using pseudo-Anosovs as in Section \ref{infdiv2}.

In both cases this gives an uncountable collection of independent divisible elements $\{\bar{f_{a}}\}_{a\in\R}$ in $H_{1}(\overline{\PMCGc(S)};\Z)$. Let $A$ be the minimal divisible subgroup containing $\{\bar{f_{a}}\}_{a\in\R}$. We can now apply the Structure Theorem of Divisible Groups. First we recall that a \textbf{quasicyclic group} is a group isomorphic to the group of $p^{n}$th complex roots of unity for all $n$ and for some prime $p$. Note that these groups are all torsion. 

\begin{theorem} [\cite{Fuchs1970}, Theorem 23.1] \label{divstructure}
    Any divisible group $D$ is a direct sum of quasicyclic and full rational groups. The cardinal numbers of the sets of quasicyclic components and $\Q$'s form a complete and independent system of invariants for $D$.
\end{theorem}
Every element in the collection $\{\bar{f_{a}}\}_{a\in\R}$ is torsion free since any power is non-trivial in the abelianization by Theorem \ref{mainthm2}. Therefore we see that $A$ has uncountably many torsion-free elements and so must contain a subgroup isomorphic to $\oplus_{2^{\aleph_{0}}}\Q$. Next we use the following.

\begin{theorem} [\cite{Fuchs1970}, Theorem 21.3] \label{abeliandecomp}
    Every abelian group $A$ is the direct sum $A = D \oplus C$ where $D$ is divisible and $C$ has no divisible subgroups other than the identity. 
\end{theorem}

\begin{proof}[Proof of Theorem \ref{divsub}]
    Write $H_{1}(\PMCGcc{S};\Z) = D \oplus C$ where $D$ is divisible and $C$ as in Theorem \ref{abeliandecomp}. Next we can further decompose $D = Q \oplus T$ where $Q$ is a direct sum of $\Q$'s and $T$ is torsion. By the above discussion and Theorem \ref{divstructure} we have $Q = \oplus_{2^{\aleph_{0}}}\Q$. Note that $Q$ cannot be any larger since the cardinality of $\PMCGcc{S}$ is $2^{\aleph_{0}}$. Letting $B = T \oplus C$ finishes the proof. 
\end{proof}


\section{Torelli group} \label{sec:torelli}

The \textbf{Torelli group}, $\cI(S)$, is the kernel of the natural homomorphism $\MCG(S) \rightarrow \Aut(H_{1}(S;\Z))$. The Torelli group has been widely studied in the finite-type case. In particular, Johnson in \cite{Johnson1985} explicitly computed the abelianization of the Torelli group when $S$ is a finite-type surface of genus at least 3 and with 1 boundary component. We will see that all of our arguments in the previous sections can be carried out in the Torelli group to obtain the same results. 

\begin{theorem} \label{torelli2}
    Let $S$ be an infinite-type surface. Then, $H_{1}(\cI(S);\Z) = \oplus_{2^{\aleph_{0}}}\Q \oplus B$ where all divisible subgroups of $B$ are torsion.
\end{theorem}

The case of the infinite-type surface with one end is also handled via a Birman Exact Sequence argument in the appendix. 

In \cite{AGKMTW2019} the authors found a topological generating set for $\cI(S)$ when $S$ is infinite type. 

\begin{theorem} [\cite{AGKMTW2019}, Corollary 2]
    Let $S$ be a connected oriented surface of infinite type. Then $\cI(S)$ is topologically generated by separating twists and bounding-pair maps. 
\end{theorem}

We thus immediately see that mapping classes of the form used in Theorem \ref{mainthm2} are contained within $\cI(S)$. This gives us nontrivial elements in $H_{1}(\cI(S);\Z)$ when $S$ has more than one end. We have to be a little bit more careful when it comes to finding divisible elements in the abelianization. Just as in the low-genus case we no longer can be sure that individual Dehn twists are contained in $[\cI(S),\cI(S)]$. However, we can use the exact same arguments as in Section \ref{infdiv2} to apply Theorem \ref{mainthm3} to obtain divisible elements in $H_{1}(\cI(S);\Z)$. To do this we simply apply Lemma \ref{thurstonconst} to a pair of separating, filling curves. Theorem \ref{torelli2} now follows by the same argument as in Section \ref{proofofC}.
 
\begin{remark}
    This theorem did not depend on the fact that we were considering the Torelli group. The conclusion of Theorem \ref{torelli2} holds for any subgroups of $\PMCGcc{S}$ that contain sufficiently many elements of the form found in Theorems \ref{mainthm2} or \ref{mainthm3}.
\end{remark}

\subsection{Johnson homomorphism}

Johnson makes use of the Johnson and Birman-Craggs-Johnson homomorphisms to explicitly compute $H_{1}(\cI(S);\Z)$ in the finite-type setting. We can extend the Johnson homomorphism to the infinite-type setting and use it to find an indivisible copy of $\oplus_{2^{\aleph_{0}}}\Z$ in $H_{1}(\cI(S);\Z)$ provided that $S$ has infinite genus.

Let $S$ be an infinite-type surface, $H = H_{1}(S;\Z)$, and $a \in H$ a primitive element. Represent $a$ by an oriented multicurve $\mu$ on the surface $S$. Given $f \in \cI(S)$ and a representative homeomorphism $\phi$ of $f$ let $M_{\phi}$ be the mapping torus of $\phi$. The cylinder $C = \mu \times [0,1]$ maps into $M_{\phi}$. Since $f \in \cI(S)$ we have that $\phi(\mu)$ is homologous to $\mu$ so that there is an immersed surface in $S \times \{0\} \subset M_{\phi}$ that closes up the cylinder $C$ to a surface $S_{a}$ in $M_{\phi}$. Note that since $S$ has at least one end the choice of this surface is unique. $S_{a}$ gives rise to a homology class $[S_{a}] \in H_{2}(M_{\phi};\Z)$. By Poincar\`{e} duality this gives a class $[S_{a}] \in H^{1}_{c}(M_{\phi};\Z)$, the first cohomology with compact support of $M_{\phi}$.

Given a triple $a \wedge b \wedge c \in \wedge^{3} H$, the third exterior power of $H$, we get an element $[S_{a}] \smallsmile [S_{b}] \smallsmile [S_{c}] \in H^{3}_{c}(M_{\phi};\Z)$. Finally, we can pair this homology class with the fundamental class of $M_{\phi}$ in locally finite homology to obtain an element of $\Z$. This gives a homomorphism, the \textbf{Johnson homomorphism},

\begin{align*}
    \tau : \cI(S) \rightarrow \operatorname{Hom}(\wedge^{3}H, \Z).
\end{align*}

We refer the reader to \cite{Hatcher2002} and \cite{Geoghegan2007} for background on locally finite homology. 

Alternatively, $\tau(f)(a \wedge b \wedge c)$, can be thought of as the triple algebraic intersection $S_{a} \cap S_{b} \cap S_{c}$. Also, just as in the finite-type case the Johnson homomorphism satisfies a \textbf{naturality property}: for $f \in \cI(S)$, $h \in \MCG(S)$, and $a \wedge b \wedge c \in \wedge^{3}H$ we have 
\begin{align*}
    \tau(hfh^{-1})(a\wedge b \wedge c) = \tau(f)(h^{-1}_{*}(a) \wedge h^{-1}_{*}(b) \wedge h^{-1}_{*}(c)).
\end{align*}

Note that since $\operatorname{Hom}(\wedge^{3}H, \Z)$ is abelian, $\tau$ factors through the abelianization of $\cI(S)$ and so nontrivial elements in the image of $\tau$ give rise to nontrivial elements in $H_{1}(\cI(S);\Z)$. We next will examine the image of $\tau$ when applied to bounding-pair maps.

\begin{remark}
    Just as in the finite-type case it can be shown that $\tau$ is trivial on any separating twist. In fact, $\tau$ is trivial on any infinite product of separating twists, provided the infinite product actually defines a mapping class. This triviality follows by building and applying the resulting homomorphism to a \emph{geometric homology basis} as defined in \cite{FHV2019}. The Johnson homomorphism can thus be seen as capturing some new information about $H_{1}(\cI(S);\Z)$ not coming from the previous constructions in this paper. 
\end{remark}

\begin{lemma} \label{bpimage}
    Let $\{\alpha,\beta\}$ be a bounding pair. 
    \begin{enumerate}[(i)]
    \item
        $\tau(T_{\alpha}T_{\beta}^{-1})$ is non-zero. 
    \item
        Let $B$ be the subsurface of $S$ with boundary $\alpha \cup \beta$. If $a', b', c'$ is a triple of curves that do not intersect $B$, then $\tau(T_{\alpha}T_{\beta}^{-1})(a' \wedge b' \wedge c') =0$.
    \end{enumerate}
\end{lemma}

\begin{proof}
    \begin{figure}
	    \centering
	    \def\svgwidth{\columnwidth}
	        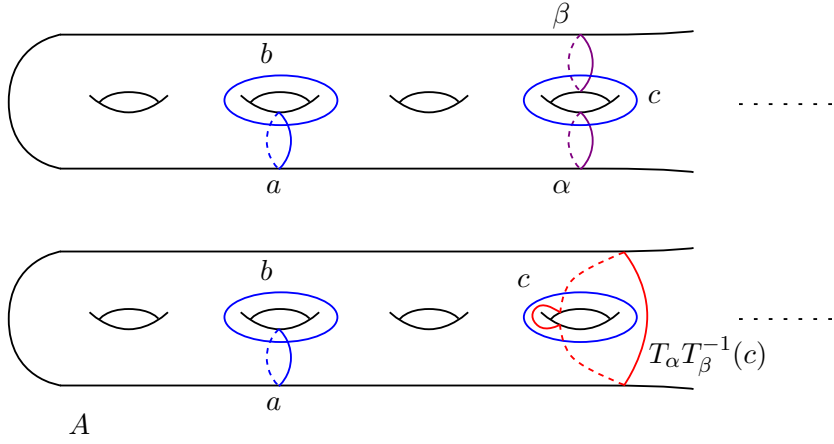
	        \caption{A generic bounding pair $\{\alpha,\beta\}$ and curves used to detect the non-triviality of $\tau(T_{\alpha}T_{\beta}^{-1})$.}
        \label{fig:johnsonbp}
    \end{figure}
    
    The change of coordinates principle and naturality property allows us to consider a standard bounding pair $\{\alpha,\beta\}$ as in Figure \ref{fig:johnsonbp}. Consider the triple of curves $a \wedge b \wedge c$ in Figure \ref{fig:johnsonbp}. Here $c$ and $T_{\alpha}T_{\beta}^{-1}(c)$ cobound the surface, A, on the left hand side of the figure. Also, $T_{\alpha}T_{\beta}^{-1}$ fixes $a$ and $b$ pointwise. Thus we see that $\tau(T_{\alpha}T_{\beta}^{-1})(a\wedge b \wedge c) = 1$, proving (i).
    
    For (ii) we simply note that $T_{\alpha}T_{\beta}^{-1}$ fixes each of $a',b',$ and $c'$ so that in the mapping torus each corresponding cylinder closes up. Thus the three corresponding subsurfaces have trivial triple intersection. 
\end{proof}

This allows us to build uncountably many linearly independent elements in the image of $\tau$ provided that $S$ has infinite genus. 
    
\begin{proposition} \label{johnsonunctbleimage}
    Let $S$ be a surface with infinite genus. Then the image of $\tau:\cI(S) \rightarrow \operatorname{Hom}(\wedge^{3}H;\Z)$ has uncountably many linearly-independent elements.
\end{proposition}

\begin{proof}
    If $f$ is a bounding-pair map we let $S_{f}$ be the finite-type surface that the corresponding bounding pair cobound. Let $\{f_{i}\}_{i\in\N}$ be a sequence of bounding-pair maps such that $\{S_{f_{i}}\}_{i\in\N}$ is a pairwise disjoint sequence of subsurfaces of $S$. Now we apply the same trick as in Section \ref{sec:uncountably} to obtain for each $a\in\R$ an infinite subset $\Lambda_{a}$ of $\N$ such that $\Lambda_{a}\cap\Lambda_{b}$ is finite for all $a\neq b$. For each $a\in \R$ let
    \begin{align*}
        f_{a} = \prod_{i \in \Lambda_{a}} f_{i}.
    \end{align*}
    
    We claim that the collection $\{\tau(f_{a})\}_{a\in \R}$ is linearly independent. We first check that $\tau(f_{a}) \neq \tau(f_{b})$ for all $a \neq b$. Let $i_{a} \in \Lambda_{a} \setminus \Lambda_{b}$ and $i_{b} \in \Lambda_{b} \setminus \Lambda_{a}$. Consider two triples of curves $x_{i_{a}}\wedge y_{i_{a}} \wedge z_{i_{a}}$ and $x_{i_{b}}\wedge y_{i_{b}} \wedge z_{i_{b}}$ where $x_{i_{a}}$ and $y_{i_{a}}$ are two curves contained in $S_{i_{a}}$ intersecting once and $z_{i_{a}}$ is a curve intersecting each of the bounding-pair curves making up $f_{i_{a}}$ that is disjoint from $S_{i_{b}}$, likewise for the other triple of curves. 
    
    Then by Lemma \ref{bpimage} we have the following. 
    \begin{align*}
        \tau(f_{a})(x_{i_{a}}\wedge y_{i_{a}} \wedge z_{i_{a}}) &= 1, \\
        \tau(f_{b})(x_{i_{a}}\wedge y_{i_{a}} \wedge z_{i_{a}}) &= 0, \\
        \tau(f_{a})(x_{i_{b}}\wedge y_{i_{b}} \wedge z_{i_{b}}) &= 0, \\
        \tau(f_{b})(x_{i_{b}}\wedge y_{i_{b}} \wedge z_{i_{b}}) &= 1.
    \end{align*}
    
    Finally, given any finite linear combination of such maps we will always be able to find such a triple that evaluates to something non-zero since any finite collection of the $\Lambda_{a}$ has finite intersection. 
\end{proof}

Since $\tau$ must factor through $H_{1}(\cI(S);\Z)$ this proves that the abelianization of $\cI(S)$, unlike the abelianization of $\PMCGcc{S}$, also contains many indivisible copies of $\Z$ when $S$ has infinite genus. 

\begin{proof}[Proof of Theorem \ref{torelli}]
    The first statement is the content of Theorem \ref{torelli2}. Assume that $S$ has infinite genus. We first note that $\operatorname{Hom}(\wedge^{3}H;\Z) \cong \Z^{\N}$ does not contain any divisible elements. We must then have that all divisible elements in $H_{1}(\cI(S);\Z)$ are contained in the kernel of $\tau$. Therefore, given the splitting $H_{1}(\cI(S);\Z) = \oplus_{2^{\aleph_{0}}}\Q \oplus B$ as in Theorem \ref{torelli2} we have that $\tau(H_{1}(\cI(S);\Z)) = \tau(B)$. By Proposition \ref{johnsonunctbleimage} we see that the image of $\tau$ contains a copy of $\oplus_{2^{\aleph_{0}}}\Z$. Finally, this is a free abelian group and so lifts to a copy of $\oplus_{2^{\aleph_{0}}}\Z$ in $B$. 
\end{proof}
    

\section{Discontinuous homomorphisms}

In this section we give counterexamples to automatic continuity in the setting of the closure of the compactly-supported mapping class group using Theorem \ref{divsub}. This application was pointed out to the author by Ryan Dickmann and in conversations with Paul Plummer, Jes\'{u}s Hern\'{a}ndez Hern\'{a}ndez, Ryan Dickmann, and Mladen Bestvina. In this section we write $\mathfrak{c} = 2^{\aleph_{0}}$ for the cardinality of the continuum.

A topological group is said to be \textbf{Polish} if it is separable and completely metrizable. In \cite{APV2020} the authors show that for an infinite-type surface, $S$, $\MCG(S)$ is Polish and hence so are all closed subgroups including $\PMCG(S)$ and $\PMCGcc{S}$.

\begin{definition}
    We say that a Polish group $G$ has \textbf{automatic continuity} if every homomorphism from $G$ to a separable topological group is necessarily continuous. 
\end{definition}

\begin{theorem}
    Let $S$ be an infinite-type surface. There exists $2^{\mathfrak{c}}$ discontinuous homomorphisms from $\PMCGcc{S}$ to $\Q$.
\end{theorem}

\begin{proof}
     Since $H_{1}(\overline{\PMCGc(S)};\Z)$ has a direct summand isomorphic to $\displaystyle\oplus_{\mathfrak{c}}\Q$ we have $2^{\mathfrak{c}}$ nontrivial homomorphisms from $H_{1}(\overline{\PMCGc(S)};\Z)$ to $\Q$. 
    
    By pre-composing each of these with the quotient homomorphism $\PMCGcc{S} \rightarrow H_{1}(\overline{\PMCGc(S)};\Z)$ we have $2^{\mathfrak{c}}$ nontrivial homomorphisms from $\PMCGcc{S}$ to $\Q$. However, since $\PMCGcc{S}$ is separable only $\mathfrak{c}$ of these can be continuous. Note that if $S$ has at least three genus then in fact none of these maps are continuous since $\PMCGc(S)$ (a dense set) is contained in the kernel of each one by the theorems of Powell and Harer. 
\end{proof}

Note that when $S$ has at most one end accumulated by genus we have $\PMCGcc{S} = \PMCG(S)$. Thus this theorem gives discontinuous homomorphisms with domain the full pure mapping class group in this setting. When $S$ is the Loch Ness monster (one end accumulated by genus) we get a discontinuous homomorphism with domain the full mapping class group. This is in contrast to the sphere minus a Cantor set for which it is known that the full mapping class group has automatic continuity \cite{Mann2020} and is uniformly perfect \cite{CalegariBlog}. 


\section{Elements not in $H_{1}(\PMCGcc{S};\Z)$ and other possible nontrivial elements}

In this section we give examples of elements in $\PMCGcc{S}\setminus \PMCGc{S}$ that are trivial in the abelianization and pose some questions about other possible nontrivial elements. 

\begin{proposition} \label{trivialelem}
    Let $S$ be an infinite-type surface. Suppose $f \in \PMCGcc{S}$ can be written as $f = \prod_{i=1}^{\infty}f_{i}$ where each $f_{i} \in \PMCGc(S)$ with $\supp(f_{i}) \subset K_{i}$ where each $K_{i}$ is a finite-type subsurface with $K_{i} \cap K_{j} =\emptyset$ if $i \neq j$. Furthermore, suppose that each $f_{i}$ can be written as a product of commutators in $\PMCG(K_{i})$ and that their commutator lengths are uniformly bounded by $N>0$. Then $f$ can be written as a product of $N$ commutators. 
\end{proposition}

\begin{proof}
    For each $i$ write $f_{i} = [g_{i_{1}},g_{i_{2}}]\cdots [g_{i_{2N-1}},g_{i_{2N}}]$ where $g_{i_{j}} \in \PMCG(K_{i})$. We allow for some of the $g_{i_{j}}$ to be the identity if $f_{i}$ has commutator length less than $N$. Thus we have
    \begin{align*}
        f = \prod_{i=1}^{\infty} \prod_{j=1}^{2N-1}[g_{i_{j}},g_{i_{j+1}}].
    \end{align*}
    Since we have $K_{i} \cap K_{j} = \emptyset$ we can rearrange this product to write
    \begin{align*}
        f = \prod_{j=1}^{2N-1}\left[\prod_{i=1}^{\infty}g_{i_{j}}, \prod_{i=1}^{\infty}g_{i_{j+1}}\right].
    \end{align*}
\end{proof}

An example of this is an infinite product of uniformly-bounded powers of commuting Dehn twists. We now ask whether the converse holds. 

\begin{question}
    Let $S$ be an infinite-type surface. Suppose $f \in \PMCGcc{S}$ can be written as $f = \prod_{i=1}^{\infty}f_{i}$ where each $f_{i} \in \PMCGcc{S}$ with $\supp(f_{i}) \cap \supp(f_{j}) = \emptyset$ for all $i\neq j$. Furthermore, suppose that each $f_{i}$ can be written as a product of commutators in $\PMCG(\supp(f_{i}))$ with unbounded commutator lengths. Then is $f$ nontrivial in $H_{1}(\PMCGcc{S};\Z)$?
\end{question}

Each component of $f$ having unbounded commutator length was the inspiration for Theorem \ref{mainthm2} and each element we construct satisfies the hypotheses of the question. However, our techniques relied heavily on the fact that each component is a power of the same mapping class on homeomorphic subsurfaces. So far we do not know how to get a large lower bound on a quasimorphism purely from the fact that the commutator lengths of the components grow. 

Our technique was also only able to detect torsion-free elements. This begs a second question.

\begin{question}
    Let $S$ be an infinite-type surface. Are there torsion elements in $H_{1}(\PMCGcc{S};\Z)$?
\end{question}


\appendix
\section{Appendix: The Loch Ness monster}

\begin{center}
    Ryan Dickmann and George Domat
\end{center}

In this appendix we prove that the mapping class group of the Loch Ness Monster surface is also not perfect. Note that for this surface and its once-punctured variant $\MCG(S)$, $\PMCG(S)$, and $\overline{\PMCGc(S)}$ are all the same. 

\begin{theorem} \label{loch}
    Let $L$ be the surface with one end, no boundary components, and infinite genus. Then $\MCG(L)$ is not perfect. In fact, $H_{1}(\MCG(L);\Z) = \oplus_{2^{\aleph_{0}}}\Q \oplus B$ where all divisible subgroups of $B$ are torsion.
\end{theorem}

To prove this we make use of the Birman Exact Sequence for infinite-type surfaces. We could not find a discussion of the infinite-type case in the literature so we present one here.

The proof is identical to the standard proof in \cite{FM11}. One only needs to check that $\pi_1(\Homeo^{+}(S))$ is trivial in the infinite-type case. The result then follows from the long exact sequence of homotopy groups given by the fiber bundle   
\begin{align*}
         \Homeo^{+}(S,x) \rightarrow \Homeo^{+}(S) \rightarrow S .
\end{align*} 

Here $\Homeo^{+}(S)$ is equipped with the compact-open topology. One can verify from the standard proof that this is indeed a fiber bundle in the infinite-type case as well. It was shown in \cite{YT2000} that the connected component of the identity in $\Homeo^{+}(S)$ is homotopy equivalent to a point for general non-compact 2-manifolds minus some degenerate finite-type cases. 

\begin{theorem}[Birman Exact Sequence]
    Let $S$ be a surface of negative Euler characteristic or infinite type. Let $(S,x)$ be the surface obtained from $S$ by adding a marked point $x$ in the interior of $S$. Then there is an exact sequence:
    \begin{align*}
        1 \rightarrow \pi_{1}(S,x) \rightarrow \MCG(S,x) \rightarrow \MCG(S) \rightarrow 1
    \end{align*}
\end{theorem}

\begin{proof} [Proof of Theorem \ref{loch}]
We first note that by applying the same abelian group theory argument as in Section \ref{proofofC} it suffices to show that $H_{1}(\MCG(L);\Z)$ contains a copy of $\oplus_{2^{\aleph_{0}}}\Q$. The general proof fails in the case of the Loch Ness Monster because we do not have end-separating curves. For the Loch Ness Monster with a puncture we do now have a separating principle exhaustion and the methods in the paper show that $H_{1}(\MCG(L,x);\Z)$ contains a copy of $\oplus_{2^{\aleph_{0}}}\Q$. 

The fundamental group of any infinite-type surface is a free group with countably many generators \cite{Stillwell1993}. Therefore we have the following exact sequence: 
    \begin{align*}
        1 \rightarrow F_{\infty} \rightarrow \MCG(L,x) \rightarrow \MCG(L) \rightarrow 1
    \end{align*}
    
Abelianization is right exact so we get the following exact sequence of abelianizations:
    \begin{align*}
        Z_{\infty} \rightarrow H_{1}(\MCG(L, x);\Z) \rightarrow H_{1}(\MCG(L);\Z) \rightarrow 1
    \end{align*}
Here $Z_{\infty}$ is the free abelian group with countably many generators. It follows that $H_{1}(\MCG(L);\Z)$ is the quotient of an uncountable group by a countable subgroup and is therefore uncountable itself. 

In fact we can do better and find a copy of $\oplus_{2^{\aleph_{0}}}\Q$ inside $H_{1}(\MCG(L);\Z)$. Since $Z_{\infty}$ is countable we must have that 
\begin{align*}
    \oplus_{2^{\aleph_{0}}}\Q \cap \ker (H_{1}(\MCG(L, x);\Z) \rightarrow H_{1}(\MCG(L);\Z))
\end{align*}
is countable where $\oplus_{2^{\aleph_{0}}} \Q$ refers to the copy found in Section 8. Thus the image of $\oplus_{2^{\aleph_{0}}}\Q$ is a divisible group with uncountably many non-torsion elements. Then by the Structure Theorem of Divisible Groups this image must again contain a copy of $\oplus_{2^{\aleph_{0}}}\Q$. 
\end{proof}

We can also apply this method of proof to the Torelli group for the Loch Ness Monster. 

\begin{theorem}
    Let $L$ be the surface with one end and infinite genus. Then $H_{1}(\cI(L);\Z) = \oplus_{2^{\aleph_{0}}}\Q \oplus B$ where all divisible subgroups of $B$ are torsion. 
\end{theorem}

\begin{proof}
    Once again, it suffices to show that $H_{1}(\cI(L);\Z)$ contains a copy of $\oplus_{2^{\aleph_{0}}}\Q$. As in Section \ref{sec:torelli} we can find elements in $\cI(L,x)$ that give rise to a copy of $\oplus_{2^{\aleph_{0}}}\Q$ in the abelianization. We can pick these elements so that they remain in $\cI(L)$ after applying the forgetful map. Indeed, we can ensure that the curves we twist about remain separating after forgetting the marked point. Now the same counting argument as in the previous theorem gives a copy of $\oplus_{2^{\aleph_{0}}}\Q$ in $H_{1}(\cI(L);\Z)$. 
\end{proof}

\section*{Acknowledgements}
The author thanks Mladen Bestvina for numerous helpful conversations and suggestions and for patiently reading through many drafts of this paper. Thanks also to Ryan Dickmann for pointing out the case of the Loch Ness Monster and the application to automatic continuity and to  Jes\'{u}s Hern\'{a}ndez Hern\'{a}ndez, Paul Plummer, and Priyam Patel for many helpful discussions about big mapping class groups. Finally, many thanks to the referee for pointing out the application in Corollary~\ref{endomorphisms} and numerous other comments. Domat was partially supported by NSF DMS-1607236, NSF DMS-1840190, and NSF DMS-1246989.

\bibliography{bib}
\bibliographystyle{halpha}

\end{document}

%% file: laddercurves.pdf_tex
\begingroup%
  \makeatletter%
  \providecommand\color[2][]{%
    \errmessage{(Inkscape) Color is used for the text in Inkscape, but the package 'color.sty' is not loaded}%
    \renewcommand\color[2][]{}%
  }%
  \providecommand\transparent[1]{%
    \errmessage{(Inkscape) Transparency is used (non-zero) for the text in Inkscape, but the package 'transparent.sty' is not loaded}%
    \renewcommand\transparent[1]{}%
  }%
  \providecommand\rotatebox[2]{#2}%
  \ifx\svgwidth\undefined%
    \setlength{\unitlength}{466.000008bp}%
    \ifx\svgscale\undefined%
      \relax%
    \else%
      \setlength{\unitlength}{\unitlength * \real{\svgscale}}%
    \fi%
  \else%
    \setlength{\unitlength}{\svgwidth}%
  \fi%
  \global\let\svgwidth\undefined%
  \global\let\svgscale\undefined%
  \makeatother%
  \begin{picture}(1,0.18329089)%
    \put(0,0){\includegraphics[width=\unitlength,page=1]{laddercurves.pdf}}%
    \put(0.41606376,0.17008005){\color[rgb]{0,0,0}\makebox(0,0)[lb]{\smash{$\gamma_{0}$}}}%
    \put(0.52421827,0.17008005){\color[rgb]{0,0,0}\makebox(0,0)[lb]{\smash{$\gamma_{1}$}}}%
    \put(0.64537096,0.17008005){\color[rgb]{0,0,0}\makebox(0,0)[lb]{\smash{$\gamma_{2}$}}}%
    \put(0.76272228,0.17008005){\color[rgb]{0,0,0}\makebox(0,0)[lb]{\smash{$\gamma_{3}$}}}%
    \put(0.29123241,0.17008005){\color[rgb]{0,0,0}\makebox(0,0)[lb]{\smash{$\gamma_{-1}$}}}%
    \put(0.18577564,0.17008005){\color[rgb]{0,0,0}\makebox(0,0)[lb]{\smash{$\gamma_{-2}$}}}%
    \put(0,0){\includegraphics[width=\unitlength,page=2]{laddercurves.pdf}}%
  \end{picture}%
\endgroup%

%% file: DTLift.pdf_tex
\begingroup%
  \makeatletter%
  \providecommand\color[2][]{%
    \errmessage{(Inkscape) Color is used for the text in Inkscape, but the package 'color.sty' is not loaded}%
    \renewcommand\color[2][]{}%
  }%
  \providecommand\transparent[1]{%
    \errmessage{(Inkscape) Transparency is used (non-zero) for the text in Inkscape, but the package 'transparent.sty' is not loaded}%
    \renewcommand\transparent[1]{}%
  }%
  \providecommand\rotatebox[2]{#2}%
  \ifx\svgwidth\undefined%
    \setlength{\unitlength}{322.43737841bp}%
    \ifx\svgscale\undefined%
      \relax%
    \else%
      \setlength{\unitlength}{\unitlength * \real{\svgscale}}%
    \fi%
  \else%
    \setlength{\unitlength}{\svgwidth}%
  \fi%
  \global\let\svgwidth\undefined%
  \global\let\svgscale\undefined%
  \makeatother%
  \begin{picture}(1,0.36869728)%
    \put(0,0){\includegraphics[width=\unitlength,page=1]{DTLift.pdf}}%
    \put(0.53850202,0.30045864){\color[rgb]{0,0,0}\makebox(0,0)[lb]{\smash{$\tilde{\gamma}$}}}%
    \put(0.04278209,0.14762525){\color[rgb]{0,0,0}\makebox(0,0)[lb]{\smash{$\tilde{\beta}$}}}%
    \put(-0.00187779,0.06366447){\color[rgb]{0,0,0}\makebox(0,0)[lb]{\smash{$T_{\gamma}(\tilde{\beta})$}}}%
  \end{picture}%
\endgroup%

%% file: QMexample.pdf_tex
\begingroup%
  \makeatletter%
  \providecommand\color[2][]{%
    \errmessage{(Inkscape) Color is used for the text in Inkscape, but the package 'color.sty' is not loaded}%
    \renewcommand\color[2][]{}%
  }%
  \providecommand\transparent[1]{%
    \errmessage{(Inkscape) Transparency is used (non-zero) for the text in Inkscape, but the package 'transparent.sty' is not loaded}%
    \renewcommand\transparent[1]{}%
  }%
  \providecommand\rotatebox[2]{#2}%
  \ifx\svgwidth\undefined%
    \setlength{\unitlength}{456.80652824bp}%
    \ifx\svgscale\undefined%
      \relax%
    \else%
      \setlength{\unitlength}{\unitlength * \real{\svgscale}}%
    \fi%
  \else%
    \setlength{\unitlength}{\svgwidth}%
  \fi%
  \global\let\svgwidth\undefined%
  \global\let\svgscale\undefined%
  \makeatother%
  \begin{picture}(1,0.43228377)%
    \put(0,0){\includegraphics[width=\unitlength,page=1]{QMexample.pdf}}%
    \put(0.18670959,0.17743327){\color[rgb]{0,0,0}\makebox(0,0)[lb]{\smash{$x_{0}$}}}%
    \put(0.89651274,0.38584779){\color[rgb]{0,0,0}\makebox(0,0)[lb]{\smash{$hx_{0}$}}}%
    \put(0,0){\includegraphics[width=\unitlength,page=2]{QMexample.pdf}}%
    \put(0.02418953,0.33351762){\color[rgb]{0,0,0}\makebox(0,0)[lb]{\smash{$Q$}}}%
    \put(0.78625019,0.23156763){\color[rgb]{0,0,0}\makebox(0,0)[lb]{\smash{$F(h)=2$}}}%
  \end{picture}%
\endgroup%

%% file: johnsonbp.pdf_tex
\begingroup%
  \makeatletter%
  \providecommand\color[2][]{%
    \errmessage{(Inkscape) Color is used for the text in Inkscape, but the package 'color.sty' is not loaded}%
    \renewcommand\color[2][]{}%
  }%
  \providecommand\transparent[1]{%
    \errmessage{(Inkscape) Transparency is used (non-zero) for the text in Inkscape, but the package 'transparent.sty' is not loaded}%
    \renewcommand\transparent[1]{}%
  }%
  \providecommand\rotatebox[2]{#2}%
  \ifx\svgwidth\undefined%
    \setlength{\unitlength}{409.44566564bp}%
    \ifx\svgscale\undefined%
      \relax%
    \else%
      \setlength{\unitlength}{\unitlength * \real{\svgscale}}%
    \fi%
  \else%
    \setlength{\unitlength}{\svgwidth}%
  \fi%
  \global\let\svgwidth\undefined%
  \global\let\svgscale\undefined%
  \makeatother%
  \begin{picture}(1,0.44977482)%
    \put(0,0){\includegraphics[width=\unitlength,page=1]{johnsonbp.pdf}}%
    \put(0.57247397,0.25596095){\color[rgb]{0,0,0}\makebox(0,0)[lb]{\smash{$\alpha$}}}%
    \put(0.57247397,0.43473925){\color[rgb]{0,0,0}\makebox(0,0)[lb]{\smash{$\beta$}}}%
    \put(0,0){\includegraphics[width=\unitlength,page=2]{johnsonbp.pdf}}%
    \put(0.27157934,0.25596095){\color[rgb]{0,0,0}\makebox(0,0)[lb]{\smash{$a$}}}%
    \put(0.26522929,0.39273123){\color[rgb]{0,0,0}\makebox(0,0)[lb]{\smash{$b$}}}%
    \put(0.67260935,0.35072321){\color[rgb]{0,0,0}\makebox(0,0)[lb]{\smash{$c$}}}%
    \put(0,0){\includegraphics[width=\unitlength,page=3]{johnsonbp.pdf}}%
    \put(0.27157934,0.0278972){\color[rgb]{0,0,0}\makebox(0,0)[lb]{\smash{$a$}}}%
    \put(0.26522929,0.16466748){\color[rgb]{0,0,0}\makebox(0,0)[lb]{\smash{$b$}}}%
    \put(0.53514131,0.15719919){\color[rgb]{0,0,0}\makebox(0,0)[lb]{\smash{$c$}}}%
    \put(0,0){\includegraphics[width=\unitlength,page=4]{johnsonbp.pdf}}%
    \put(0.67337699,0.07806781){\color[rgb]{0,0,0}\makebox(0,0)[lb]{\smash{$T_{\alpha}T_{\beta}^{-1}(c)$}}}%
    \put(0,0){\includegraphics[width=\unitlength,page=5]{johnsonbp.pdf}}%
    \put(0.06364595,0.00344406){\color[rgb]{0,0,0}\makebox(0,0)[lb]{\smash{$A$}}}%
    \put(0,0){\includegraphics[width=\unitlength,page=6]{johnsonbp.pdf}}%
  \end{picture}%
\endgroup%